\documentclass[12pt]{article}

\usepackage[a4paper]{geometry}
\usepackage{caption2}
\usepackage{multirow}
\usepackage{enumitem}
\usepackage{mathrsfs}
\usepackage{makecell}
\usepackage{algorithm,algorithmic}
\usepackage{amsmath}
\usepackage{amssymb}
\usepackage{amsfonts}
\usepackage{graphicx}
\usepackage{epstopdf}
\usepackage{float}
\usepackage{amsthm}
\usepackage[unicode={ture},colorlinks,
            linkcolor=black,
            anchorcolor=blue,
            citecolor=green]{hyperref}
\usepackage{appendix}

\allowdisplaybreaks



\def\ge{\geqslant}

\newtheorem{thm}{Theorem}[section]
\newtheorem{lem}[thm]{Lemma}

\newtheorem{re}{Remark}[section]

\title{ Lower bound of the energy of a complex unit gain graph in terms of the matching number of its underlying graph }
\author{
	Yuxuan Li \footnote{ E-mail address: liyx@mail.bnu.edu.cn(Y. Li)} \\
	{\footnotesize    \em  Sch. Math. Sci. {\rm \&} Lab. Math. Com. Sys.,
		Beijing Normal University, Beijing, 100875,  China}
}
\date{}

\begin{document}
\maketitle
\begin{abstract}

We establish a lower bound for the energy of a complex unit gain graph in terms of the matching number of its underlying graph, and characterize all the complex unit gain graphs whose energy reaches this bound. 

\medskip
\noindent {\em AMS classification:} 05C50; 05C35; 05C75.

\noindent {\em Key words:} complex unit gain graph; energy; matching number.
\end{abstract}
\section{Introduction}\label{s1}

All graphs considered in this paper are finite, undirected and simple. Let $G$ be a graph with vertex set $V(G)$ and edge set $E(G)$. Let $\vec{E}(G)$ be the set $\{e_{uv},e_{vu}~|~uv\in E(G)\}$, where $e_{uv}$ denotes the ordered pair $(u,v)$. We use $\mathbb{T}$ to denote the the multiplicative group consisting of all the complex units, i.e. $\mathbb{T}=\{z\in \mathbb{C}~|~|z|=1\}$. Suppose $\phi:\vec{E}(G)\to \mathbb{T}$ is an arbitrary mapping with the property $\phi(e_{uv})=\phi(e_{vu})^{-1}$. Then $\Phi=(G,\mathbb{T},\phi)$, with $G$ as its underlying graph, is called a complex unit gain graph (or $\mathbb{T}$-gain graph), which was introduced by Reff in \cite{R}. We refer to the elements in $\mathbb{T}$ as gains and $\phi$ as the gain function of $\Phi$. Note that a graph $G$ is just a complex unit gain graph with all the ordered pairs in $\vec{E}(G)$ having $1$ as their gains, and we denote it by $(G,\mathbb{T},1)$. If there is a cycle $C$ of $G$ consisting of edges $v_{1}v_{2},v_{2}v_3,\ldots, v_kv_1$, it can be given two directions: $C_1=v_{1}v_{2}\cdots v_{k}v_{1}$ and $C_2=v_{k}v_{k-1}\cdots v_{1}v_{k}$. The gain of $C_1$ in $\Phi$ is  defined to be $\phi(C_1) = \phi(e_{v_{1}v_{2}})\phi(e_{v_{2}v_{3}}) \cdots \phi(e_{v_{k-1}v_{k}})\phi(e_{v_{k}v_{1}})$. Similarly, we can define $\phi(C_2)$ and have $\phi(C_2)=\phi(C_1)^{-1}$. If $\phi(C_1)=\phi(C_2)=1$, we say the cycle $C$ is neutral in $\Phi$ without a mention of the direction. If every cycle of $G$ is neutral in $\Phi$, we say $\Phi$ is balanced. Complex unit gain graphs have caused attention in recent years. For more information, see \cite{HHD,MMA,YQT}.

The adjacency matrix $A(\Phi)$ of $\Phi=(G,\mathbb{T},\phi)$ is the $n\times n$ complex matrix $(a_{ij})$ with $a_{ij}=\phi(e_{v_iv_j})$ if $e_{v_iv_j}\in \vec{E}$ and $0$ otherwise, where $n$ is the order of $G$. Clearly, $A(\Phi)$ is Hermitian and all its eigenvalues are real. The energy of $\Phi$, denoted by $\mathcal{E}(\Phi)$, is defined to be the sum of the absolute values of its eigenvalues.

There are amounts of literature investigating the bounds on the energy of a graph in terms of other parameters.  In \cite{AGZ}, Akbari et al. proved the rank of a graph is a sharp lower bound of its energy. Wang and Ma \cite{WM} gave sharp bounds of graph energy in terms of vertex cover number and characterized all the extremal graphs attaining these bounds. Wong et al. \cite{WWC} proved $\mathcal{E}(G)\ge 2\mu(G)$, where $\mu(G)$ is the matching number of $G$, and partially characterized the extremal graphs. These results have already been extended to oriented graphs and mixed graphs in \cite{TW,WL}.

In this paper, we establish a lower bound for the energy of a complex unit gain graph
in terms of the matching number of its underlying graph, and characterize
all the complex unit gain graphs whose energy reaches this lower bound. Our result generalizes the corresponding results on graphs \cite{WWC}, oriented graphs \cite{TW} and mixed graphs in \cite{WL}.

\begin{thm}\label{thm 1.1}
  Let $\Phi=(G,\mathbb{T},\phi)$ be a complex unit gain graph and $\mu(G)$ the matching number of $G$. Then 
  \begin{equation}\label{eq 1}
    \mathcal{E}(\Phi)\ge 2\mu(G).
  \end{equation}
  Equality holds if and only if $\Phi$ is balanced and $G$ is the disjoint union of some regular complete bipartite graphs, together with some isolated vertices.
\end{thm}

\section{Preliminaries}\label{s2}

In this section, we shall introduce some notations and lemmas on complex unit gain graphs.

Recall that the degree of a vertex $u$ of $G$ is the number of its neighbors, i.e., the number of the vertices which are adjacent to $u$. If the degree of $u$ is $1$, we call it a pendant vertex of $G$. A graph $H$ is a subgraph of $G$ if $V(H)\subseteq V(G)$ and $E(H)\subseteq E(G)$. Further, if two vertices in $V(H)$ are adjacent in $H$ if and only if they are adjacent in $G$, $H$ is called an induced subgraph of $G$. For any subgraph $G_1$ of $G$, we define the subgraph $\Phi_1=(G_1,\mathbb{T},\phi)$ of $\Phi=(G,\mathbb{T},\phi)$ by restricting $\phi$ to $\overrightarrow{E}(G_1)=\{e_{uv},e_{vu}~|~uv\in E(G_1)\}$.

Let $G[V_1]$ (resp. $G[E_1]$) be the subgraph of $G$ induced by $V_1 \subseteq V(G)$ (resp. $E_1 \subseteq E(G)$). We use $G-V_1$ to denote $G[ \overline{V_1} ]$, where $\overline{V_1}=V(G)\setminus  V_1$. For any induced subgraph $H$ of $G$, we simply denote $G-V(H)$ by $G-H$ and call it the complement of $H$ in $G$. We write $G = H \oplus (G-H)$ when no edges in $G$ join the induced subgraph $H$ and its complement $G-H$. For a nonempty set $S\subseteq E(G)$, let $G-S$ be the spanning subgraph obtained from $G$ by deleting the edges in $S$. If there exists an induced subgraph $K$ such that $G-S=K\oplus (G-K)$, $S$ is called an edge cut of $G$. 

 Similar to \cite[Theorems 3.4 and 3.6]{DS}, we have the following lemma.

\begin{lem}\label{lem 3.1}
   Let $\Phi=(G,\mathbb{T},\phi)$ be a complex unit gain graph and $S$ an edge cut of $G$. Then $\mathcal{E}(G-S,\mathbb{T},\phi) \leq \mathcal{E}(\Phi)$. Further, if $G[S]$ is a star, $\mathcal{E}(G-S,\mathbb{T},\phi)< \mathcal{E}(\Phi)$.
\end{lem}

A matching $M$ of $G$ is an edge subset such that no two edges in $M$ share a common vertex. If $u$ is incident to some edge in $M$, $u$ is said to be saturated by $M$. Vertices which are not incident to any edge in $M$ are  unsaturated by $M$. A maximum matching of $G$ is a matching which contains the largest possible number of edges. The size of a maximum matching is known as the matching number of $G$, denoted by $\mu(G)$. $M$ is called a perfect matching of $G$ if every vertex of $G$ is saturated by $M$.

Similar to \cite[Theorem 1.1 {\rm (i)}]{WWC}, Lemma \ref{lem 3.1} implies the inequality (\ref{eq 1}) in Theorem \ref{thm 1.1}. Using this inequality, the following Lemmas \ref{lem 3.3}-\ref{lem 3.5} can be proved by similar methods in \cite{TW} and \cite{WL}.

\begin{lem}\label{lem 3.3}
   Let $\Phi=(G,\mathbb{T},\phi)$ be a complex unit gain graph with at least $3$ vertices. If $G$ is connected and has a pendant vertex, then $\mathcal{E}(\Phi)> 2\mu(G)$.
\end{lem}

\begin{lem}\label{lem 3.4}
   Let $\Phi=(\tilde{C_6},\mathbb{T},\phi)$ be a complex unit gain graph, where $\tilde{C_6}$ is obtained from a $6$-cycle $C_6 = v_1v_2v_3v_4v_5v_6v_1$ by adding an edge $v_2v_5$. Then $\mathcal{E}(\Phi)> 6$.
\end{lem}

\begin{lem}\label{lem 3.2}
   Let $\Phi=(G,\mathbb{T},\phi)$ be a complex unit gain graph. Suppose $G_1$ is an induced subgraph of $G$ with $\mu(G)= \mu(G_1) + \mu(G-G_1)$. If $\mathcal{E}(\Phi)=2\mu(G)$, then $\mathcal{E}(G_1,\mathbb{T},\phi)=2\mu(G_1)$ and $G_1$ is not $P_4$ or $\tilde{C_6}$. 
\end{lem}

\begin{lem}\label{lem 3.5}
   Let $\Phi=(G,\mathbb{T},\phi)$ be a complex unit gain graph without isolated vertices. If $\mathcal{E}(\Phi) = 2\mu(G)$, then $G$ has a perfect matching.
\end{lem}

\section{Proof of Theorem \ref{thm 1.1}}\label{s3}
 
To prove Theorem \ref{thm 1.1}, we need the following lemmas.

\begin{lem}\label{lem 3.6}
   Let $\Phi=(G,\mathbb{T},\phi)$ be a complex unit gain graph, where $G$ is a connected bipartite graph with at least two vertices. If $\mathcal{E}(\Phi) = 2\mu(G)$, $G$ is a regular complete bipartite graph.
\end{lem}
\begin{proof}
   Assume that $G$ has $n$ vertices and its two partite sets are $X,Y$. Since $\mathcal{E}(\Phi) = 2\mu(G)$, by Lemma \ref{lem 3.5} we know that $G$ has a perfect matching, say $M$. Thus, $n$ is even and $|X|=|Y|= \mu(G) = n/2$.

   Next we prove $G$ is complete by induction on $n$. If $n=2$, $G$ is the complete graph of order $2$. Thus the result holds clearly. We assume that the result holds for complex unit gain graphs of order at most $n-2$ $(n\ge 4)$. In what follows we suppose that $\Phi=(G,\mathbb{T},\phi)$ is an $n$-vertex complex unit gain graph with $\mathcal{E}(\Phi) = 2\mu(G)$ and $G$ is a connected bipartite graph. 

   Let $X =\{x_1, x_2, \ldots , x_{n/2}\}$ and $Y =\{y_1, y_2, \ldots , y_{n/2}\}$. Suppose on the contrary that $G$ is incomplete. Then there exist vertices $x_1 \in X$, $y_1 \in Y$ such that $x_1$ is not adjacent to $y_1$ in $G$. Suppose that $x_1$ is $M$-saturated by edge $x_1y_2$ and $y_1$ is $M$-saturated by edge $x_2y_1$. If $n=4$, $x_2$ must be adjacent to $y_2$ as $G$ is connected and thus $G \cong P_4$ with two pendant vertices $x_1, y_1$. By Lemma \ref{lem 3.3}, we have $\mathcal{E}(\Phi)> 2\mu(G)$, which is a contradiction. Thus, $x_1$ must be adjacent to $y_1$ and hence $G$ is complete. 

   If $n\ge 6$, let $G_1= G[\{x_1,x_2,y_1,y_2\}]$, $G_2 = G-\{x_1,x_2,y_1,y_2\}$ and $S$ the edge cut of $G$ such that $G-S=G_1 \oplus G_2$. Note that $\mu(G) = \mu(G_1) + \mu(G_2)$. Denote $(G_1,\mathbb{T},\phi)$ and $(G_2,\mathbb{T},\phi)$ by $\Phi_1$ and $\Phi_2$, respectively. By Lemma \ref{lem 3.2}, we have $\mathcal{E}(\Phi_1) = 2\mu(G_1)$, $\mathcal{E}(\Phi_2)= 2\mu(G_2)$ and $G_1$ is not $P_4$. Thus, $x_2$ is not adjacent to $y_2$ in $G$ and $G_1\cong 2K_2$. Let $G_1^1=G[\{x_1,y_2\}]$ and $G_1^2=G[\{x_2,y_1\}]$. Thus $G_1=G_1^1\oplus G_1^2$.

   Assume that $G_2$ has $\omega$ connected components, denoted by $G_2^1,G_2^2, \ldots , G_2^\omega$. As $G_2$ has a perfect matching, each of these connected components is non-trivial and has a perfect matching. By the inequality (\ref{eq 1}), $$2\mu(G_2)=\mathcal{E}(\Phi_2)=\sum_{j=1}^\omega \mathcal{E}(G_2^j,\mathbb{T},\phi)\ge\sum_{j=1}^\omega 2\mu(G_2^j) = 2\mu(G_2).$$Hence, $\mathcal{E}(G_2^j,\mathbb{T},\phi) = 2\mu(G_2^j)$ for $j=1,2,\ldots,\omega$. By induction hypothesis, we obtain that each connected component $G_2^j$ is a regular complete bipartite graph.

   If $x,y\in V(G)$ are adjacent in $G$, we write $x\sim y$. According to Lemma \ref{lem 3.2}, if $x_1\sim y_0$ (resp. $x_2\sim y_0$) for any $y_0\in V(G_2^j)\cap Y$, $y_2\sim x$ (resp. $y_1\sim x$) for every $x\in V(G_2^j)\cap X$ and $x_1\sim y$ (resp. $x_2\sim y$) for every $y\in V(G_2^j)\cap Y$. Similarly, if $y_1\sim x_0$ (resp. $y_2\sim x_0$) for any $x_0\in V(G_2^j)\cap X$, $x_2\sim y$ (resp. $x_1\sim y$) for every $y\in V(G_2^j)\cap Y$ and $y_1\sim x$ (resp. $y_2\sim x$) for every $x\in V(G_2^j)\cap X$ and .

   We claim that $\omega=1$. Suppose on the contrary that $\omega\ge 2$. For two subsets $V_1,V_2$ of $V(G)$, let $E(V_1,V_2)=\{uv\in E(G)~|~u\in V_1,v\in V_2\}$. Since $G$ is connected, there exist two different connected components $G_2^{j_1}$ and $G_2^{j_2}$ of $G_2$ such that both $E(V(G_1^i),V(G_2^{j_1}))$ and $E(V(G_1^i),V(G_2^{j_2}))$ are non-empty sets, where $i=1$ or $2$. Without loss of generality, suppose both $E(V(G_1^1),V(G_2^1))$ and $E(V(G_1^1),V(G_2^2))$ are non-empty. For any $x_3\in V(G_2^1)\cap X$, $y_3 \in V(G_2^1)\cap Y$, $x_4\in V(G_2^2)\cap X$ and $y_4 \in V(G_2^2)\cap Y$, consider the subgraph $H$ induced by $\{x_1,x_3,x_4,y_2,y_3,y_4\}$. Clearly, $\mu(G)=\mu(H)+\mu(G-H)$ and $H\cong \tilde{C_6}$, which is a contradiction by Lemma \ref{lem 3.2}. Thus $\omega=1$.

   Now we have $G-S=G_1^1\oplus G_1^2\oplus G_2$, where $G_1^1$, $G_1^2$ and $G_2$ are connected. As $G$ is connected, both $E(V(G_1^1),V(G_2))$ and $E(V(G_1^2),V(G_2))$ are non-empty. Then consider the subgraph $K$ induced by $\{x_1,x_2,x_3,y_1,y_2,y_3\}$, where $x_3\in V(G_2)\cap X$ and $y_3\in V(G_2)\cap Y$. We have $K\cong \tilde{C_6}$ and $\mu(G)=\mu(K)+\mu(G-K)$, which is contradict with Lemma \ref{lem 3.2}. Thus, the assumption that $x_1\not \sim y_2$ is incorrect and we obtain that $G$ is a complete bipartite graph. As $|X|=|Y|$, $G$ is also regualr.
\end{proof}

Any function $\zeta:V(G)\to \mathbb{T}$ is called a switching function of $G$. Switching $\Phi=(G,\mathbb{T},\phi)$ by $\zeta$ means replacing $\phi$ by $\phi^\zeta$, which is defined as $\phi^\zeta(e_{uv}) = \zeta(u)^{-1}\phi(e_{uv})\zeta(v)$, and the resulting graph is denoted by $\Phi^\zeta = (G,\mathbb{T},\phi^\zeta)$. In this case, we say $\Phi$ and $\Phi^\zeta$ are switching equivalent, written by $\Phi\sim \Phi^\zeta$. As $A(\Phi)$ and $A(\Phi^\zeta)$ are similarity matrices, $\Phi$ and $\Phi^\zeta$ have the same energy. \cite[Lemma 5.3]{Z} shows that a complex unit gain graph $\Phi=(G,\mathbb{T},\phi)$ is balanced if and only if $\Phi$ is switching equivalent to $(G,\mathbb{T},1)$.

\begin{lem}\label{lem 3.10}
   Let $\Phi=(G,\mathbb{T},\phi)$ be a complex unit gain graph, where $G$ is a connected bipartite graph with at least two vertices. If $\mathcal{E}(\Phi) = 2\mu(G)$, then $\Phi$ is balanced.
\end{lem}
\begin{proof}
   According to Lemma \ref{lem 3.6}, we know $G$ is a regular complete bipartite. We show that $\Phi$ is balanced by induction on $n$. If $n=2$, $G$ is a tree which has no cycles and thus $\Phi$ is balanced. We assume that the result holds for complex unit gain graphs of order $n-2$ $(n\ge 4)$. Let $G$ be an $n$-vertex regular complete bipartite graph with $X=\{x_1,x_2,\ldots,x_{n/2}\}$ and $Y=\{y_1,y_2,\ldots,y_{n/2}\}$ being its partite sets and $\Phi=(G,\mathbb{T},\phi)$ a complex unit gain graph with $\mathcal{E}(\Phi)=2\mu(G)$. 

   Let $S$ be the edge cut of $G$ such that $G-S=G_1\oplus G_2$, where $G_1=G[\{x_1,y_1\}]$ and $G_2=G-\{x_1,y_1\}$. Clearly, $G_1$ is balanced, and thus there exists a switching function of $G_1$, denoted by $\zeta_1':V(G_1)\to \mathbb{T}$, such that $(G_1,\mathbb{T},\phi^{\zeta_1'})=(G_1,\mathbb{T},1)$. Define a switching function $\zeta_1:V(G)\to \mathbb{T}$ of $G$, where $\zeta_1(x_1)=\zeta_1'(x_1)$, $\zeta_1(y_1)=\zeta_1'(y_1)$ and $\zeta_1(z)=1$ for any $z\in V(G)\setminus\{x_1,y_1\}$. Switch $\Phi$ by $\zeta_1$ and denote $\Phi^{\zeta_1}$ by $\Phi_1=(G,\mathbb{T},\phi_1)$, where $\phi_1=\phi^{\zeta_1}$. Then we have $(G_1,\mathbb{T},\phi_1)=(G_1,\mathbb{T},1)$ and $\mathcal{E}(\Phi_1)=\mathcal{E}(\Phi)=2\mu(G)$. Note the fact that $\mu(G)=\mu(G_1)+\mu(G_2)$. Then according to Lemma \ref{lem 3.2}, we have $\mathcal{E}(G_2,\mathbb{T},\phi_1)=2\mu(G_2)$. 

   Note that $G_2$ is an $(n-2)$-vertex regular complete bipartite graph. Then by induction hypothesis, $(G_2,\mathbb{T},\phi_1)$ is balanced and thus there exists a switching function $\zeta_2':V(G_2)\to \mathbb{T}$ of $G_2$ such that $(G_2,\mathbb{T},\phi_1^{\zeta_2'})=(G_2,\mathbb{T},1)$. Define a new switching function $\zeta_2:V(G)\to \mathbb{T}$ of $G$, where $\zeta_2(z)=\zeta_2'(z)$ for all $z\in V(G_2)$ and $\zeta_2(x_1)=\zeta_2(y_1)=1$. Switch $\Phi_1$ by $\zeta_2$ and denote $\Phi_1^{\zeta_2}$ by $\Phi_2=(G,\mathbb{T},\phi_2)$, where $\phi_2=\phi_1^{\zeta_2}$. Clearly, $(G_1,\mathbb{T},\phi_2)=(G_1,\mathbb{T},1)$, $(G_2,\mathbb{T},\phi_2)=(G_2,\mathbb{T},1)$ and $\mathcal{E}(\Phi_2)=\mathcal{E}(\Phi_1)=\mathcal{E}(\Phi)=2\mu(G)$.

   Consider the subgraph $K$ induced by $\{x_1,y_1,x_0,y_0\}$ where $x_0$ is any vertex in $V(G_2)\cap X$ and $y_0$ is any vertex in $V(G_2)\cap Y$. Then $K$ is a $4$-cycle $x_1y_1x_0y_0x_1$ and $\phi_2(e_{x_1y_1})=\phi_2(e_{x_0y_0})=1$. Note that $\mu(G)=\mu(K)+\mu(G-K)$. By Lemma \ref{lem 3.2}, we have $\mathcal{E}(K,\mathbb{T},\phi_2)=2\mu(K)=4$. Suppose $\phi_2(e_{y_1x_0})=a$ and $\phi_2(e_{y_0x_1})=b$ where $a,b\in \mathbb{T}$. Then the characteristic polynomial of $A(K,\mathbb{T},\phi_2)$ is $f(\lambda)=\lambda^4-4\lambda^2+2-2{\rm Re}(ab)$. Let $x={\rm Re}(ab)\in [-1,1]$. The energy of $(K,\mathbb{T},\phi_2)$ is $$2\sqrt{2+\sqrt{2+2x}}+2\sqrt{2-\sqrt{2+2x}}\ge 4,$$ and the equality holds if and only if $x=1$ which implies $a=\bar{b}$.   Thus we know that $\phi_2(e_{y_1x_0})=\phi_2(e_{x_1y_0})=a$. Because of the arbitrariness of $y_0\in V(G_2)\cap Y$, we have for any $y\in V(G_2)\cap Y$, $\phi_2(e_{x_1y})=a$. Similarly, due to the arbitrariness of $x_0\in V(G_2)\cap X$, we have for any $x\in V(G_2)\cap X$, $\phi_2(e_{y_1x})=a$.

   Here we have $(G_1,\mathbb{T},\phi_2)=(G_1,\mathbb{T},1)$, $(G_2,\mathbb{T},\phi_2)=(G_2,\mathbb{T},1)$ and $\phi_2(e_{x_1y})=\phi_2(e_{y_1x})=a$ for every $y\in V(G_2)\cap Y$ and $x\in V(G_2)\cap X$. Define the third switching function $\zeta_3:V(G)\to \mathbb{T}$ of $G$, where $\zeta_3(x_1)=\zeta_3(y_1)=1$ and $\zeta_3(z)=a^{-1}$ for all $z\in V(G_2)$. Switching $\Phi_2$ by $\zeta_3$ and denote $\Phi_2^{\zeta_3}$ by $\Phi_3=(G,\mathbb{T},\phi_3)$, where $\phi_3=\phi_2^{\zeta_3}$. One can verify that all edges in $\Phi_3$ have the gain $1$. Thus we switch $\Phi$ by $\zeta_1$, $\zeta_2$ and $\zeta_3$ successively and then get $(G,\mathbb{T},1)$. By Lemma 5.3 in \cite{Z}, we obtain the desired result that $\Phi$ is balanced.
\end{proof}

Suppose $G$ and $H$ are graphs with vertex set $V(G)=\{v_1,v_2,\ldots,v_n\}$ and $V(H) =\{u_1, u_2, \ldots , u_m\}$, respectively. Then we define the Kronecker product of $\Phi=(G,\mathbb{T},\phi)$ and $H$, which is also a complex unit gain graph, denoted by
$\Phi\otimes H$. Its underlying graph is $G\otimes H$ with vertex set $\{(v_s,u_t)~|~s=1,2,\ldots,n;t=1,2,\ldots,m\}$ and edge set $\big\{(v_s,u_t)(v_{s'},u_{t'})~\big|~v_sv_{s'}\in E(G),~u_t u_{t'}\in E(H)\big\}$. The gain of $e_{(v_s,u_t)(v_{s'},u_{t'})}$ in $\Phi\otimes H$ is defined to be the gain of $e_{v_sv_{s'}}$ in $\Phi$. In particular, $\Phi\otimes K_2$ is called the complex unit gain bipartite double of $\Phi$.

Let $U=(u_{st})$ and $V$ be two matrices of order $p_1 \times p_2$ and $q_1 \times q_2$, respectively. The Kronecker product of $U$ and $V$ is defined to be $U\otimes V=(u_{st}V)$, which is a $p_1q_1 \times p_2q_2$ matrix. Note that the adjacency matrix of $\Phi\otimes H$ is $A(\Phi\otimes H) = A(\Phi) \otimes A(H)$. If the eigenvalues of $\Phi$ are $\eta_1, \eta_2, \ldots , \eta_n$ and the eigenvalues of $H$ are $\lambda_1, \lambda_2, \ldots , \lambda_m$, the eigenvalues of $\Phi\otimes H$ are $\eta_s\lambda_t$ where $s=1,2,\ldots,n$ and $t=1,2,\ldots,m$ (see \cite{B} for details).

\begin{lem}\label{lem 3.7}
   Let $\Phi=(G,\mathbb{T},\phi)$ be a complex unit gain graph whose underlying graph $G$ is connected and non-bipartite. Then we have $\mathcal{E}(\Phi)> 2\mu(G)$.
\end{lem}
\begin{proof}
   Suppose on the contrary that $\mathcal{E}(\Phi)=2\mu(G)=n$. Then by Lemma \ref{lem 3.5}, $G$ has a perfect matching. As $\mu(G)=n/2$, we know $G$ has $n$ vertices, denoted by $\{v_1,v_2,\ldots,v_{n}\}$. Suppose that $G$ has $m$ edges. Let $H$ be the complete graph with vertex set $\{u_1,u_2\}$. Then consider the Kronecker product $\Phi\otimes H$. Its underlying graph $G\otimes H$ has $2n$ vertices and $2m$ edges. Clearly, $G\otimes H$ is a bipartite graph with $X=\{(v_i,u_1)~|~i=1,2,\ldots,n\}$ and $Y=\{(v_i,u_2)~|~i=1,2,\ldots,n\}$ being its two partite sets and $\mathcal{E}(\Phi\otimes H)=2\mathcal{E}(\Phi)=2n$. Since $G$ has a perfect matching, so does $G\otimes H$ and $\mu(G\otimes H)=2\mu(G)=n$. Then $\mathcal{E}(\Phi\otimes H)=2\mu(G\otimes H)$. 

   Clearly, $G\otimes H$ is incomplete, as $(v_i,u_1)\in X$ is not adjacent to $(v_i,u_2)\in Y$ for any $i=1,2,\ldots,n$. According to Lemma \ref{lem 3.6}, we know $G\otimes H$ is not connected. 

  We claim $G\otimes H$ has only two isomorphic connected components. Suppose that $G\otimes H$ has $l$ connected components, denoted by $\Omega_1,\Omega_2,\ldots,\Omega_l$, each of which is a bipartite graph. Then by inequality (\ref{eq 1}), $2\mu(G\otimes H)=\mathcal{E}(\Phi\otimes H)=\sum_j\mathcal{E}(\Omega_j,\mathbb{T},\phi)\ge \sum_j 2\mu(\Omega_j)=2\mu(G\otimes H)$. Then $\mathcal{E}(\Omega_j,\mathbb{T},\phi)=2\mu(\Omega_j)$ for all $j=1,2,\ldots,l$. By Lemma \ref{lem 3.6}, we know that each $\Omega_j$ is a regular complete bipartite graph. 

   Suppose the two partite sets of $\Omega_1$ are $X_1=\{(x_1,u_1),(x_2,u_1),\ldots,(x_t,u_1)\}$ and $Y_1=\{(y_1,u_2),(y_2,u_2),\ldots,(y_t,u_2)\}$, where $x_i,y_i\in V(G)$ for $i=1,2,\ldots,t$. Let $X_1'=\{x_1,x_2,\ldots,x_t\}$ and $Y_1'=\{y_1,y_2,\ldots,y_t\}$. Since $\Omega_1$ is a complete bipartite graph, $x_i\sim y_j$ in $G$ for any $i,j=1,2,\ldots,t$. Thus $X_1'\cap Y_1'=\emptyset$ and $Y_1'\subseteq N(x_i)$ for each $i=1,2,\ldots,t$, where $N(x_i)$ is the set of the neighbors of $x_i$ in $G$. Suppose there exists $z\in V(G)\setminus Y_1'$ such that $x_i\sim z$ in $G$. Then $(x_i,u_1)$ must be adjacent to $(z,u_2)$ in $G\otimes H$. Hence $(z,u_2)\in Y_1$, which implies $z\in Y_1'$, a contradiction to the choice of $z$. Thus $N(x_i)=Y_1'$ for all $i=1,2,\ldots,t$. Similarly, $N(y_j)=X_1'$ for all $j=1,2,\ldots,t$. Let $X_2=\{(x_1,u_2),(x_2,u_2),\ldots,(x_t,u_2)\}$ and $Y_2=\{(y_1,u_1),(y_2,u_1),\ldots,(y_t,u_1)\}$. Then $X_2\cup Y_2$ induce another connected component of $G\otimes H$, say $\Omega_2$. 

   Consider the subgraph $G'$ of $G$ induced by $X_1'\cup Y_1'$. Clearly, $G'$ is a complete bipartite graph with $X_1'$ and $Y_1'$ being its two partite sets, and both $\Omega_1$ and $\Omega_2$ are isomorphic to $G'$. We also assert that $G'$ is a connected components of $G$ . If $G\otimes H$ has a third connected component, then $V(G)\setminus V(G')$ is not empty and thus $G$ is not connected, which is a contradiction. Here we prove the desired result that $G\otimes H$ has only two isomorphic connected components.

   From the above discuss, we also know that $G$ is a complete bipartite graph. This is contradictory with the fact that $G$ is non-bipartite. Thus the assumption $\mathcal{E}(\Phi)=2\mu(G)$ is incorrect and $\mathcal{E}(\Phi)>2\mu(G)$. 
\end{proof}

\noindent{\bf Proof of Theorem \ref{thm 1.1}}:
The inequality (\ref{eq 1}) can be proved by a similar method used in \cite[Theorem 1.1 {\rm (i)}]{WWC}. In the following, we prove the necessary and sufficient conditions for the energy of a complex unit gain graph to reach its lower bound.

 (Sufficiency) Assume $$G=(\cup_{j=1}^\omega K_{n_j,n_j})\cup (n-2\sum_{j=1}^\omega n_j)K_1,$$ where $K_{s,n-s}$ and $K_n$ are a complete bipartite graph and the complete graph of order $n$, respectively. Since $\Phi$ is balanced, we have  $$\mathcal{E}(\Phi)=\sum_{j=1}^\omega\mathcal{E}(K_{n_j,n_j},\mathbb{T},\phi)=\sum_{j=1}^\omega\mathcal{E}(K_{n_j,n_j})=\sum_{j=1}^\omega 2n_j=\sum_{j=1}^\omega 2\mu(K_{n_j,n_j})=2\mu(G).$$ 

(Necessity) Assume that $\mathcal{E}(\Phi) = 2\mu(G)$. Suppose that $G$ has $\omega$ non-trivial connected components $G_1, G_2, \ldots , G_\omega$. By the inequality (\ref{eq 1}), we obtain $$2\mu(G)=\mathcal{E}(\Phi) = \sum_{j=1}^\omega\mathcal{E}(G_j,\mathbb{T},\phi)\ge 2\mu(G_j)=2\mu(G).$$
Thus $\mathcal{E}(G_j,\mathbb{T},\phi)=2\mu(G)$ for $j=1,2,\ldots,\omega$. By Lemma \ref{lem 3.7}, each non-trivial connected component $G_j$ is bipartite. By Lemma \ref{lem 3.6} and \ref{lem 3.10}, $(G_j,\mathbb{T},\phi)$ is balanced and $G_j$ is a regular complete bipartite graph, for $j=1,2,\ldots,\omega$. Therefore, $\Phi$ is balanced and $G$ is the disjoint union of some regular complete bipartite graphs, together with some isolated vertices. ~~~~~~~~~~~~~~~~~~~~~~~~~~~~~~~~~~~~~~~~~~~~~~~~~~~~~~~~~~~~~~~~~~~~~~~~~~~~~~~~~~$\qed$

\begin{re}
{\rm 
 The results in \cite[Theorem 1.3]{WF} and \cite[Theorem 5.2]{WLM} can be extended to complex unit gain graphs, which gives a upper bound of $\mathcal{E}(\Phi)$ in terms of the rank of $\Phi$ and characterizes all the extremal complex unit gain graphs. Applying Theorem \ref{thm 1.1}, the bounds of graph energy in terms of the vertex cover number given in \cite[Theorems 3.1 and 4.2 ]{WM} can also be extended to the energy of complex unit gain graphs. However, the equality case in \cite[Theorem 3.1]{WM} follows from Perron-Frobinus Theorem, which only holds for real matrices. }
\end{re}

\section*{Acknowledgement}

The author would like to thank Professor Kaishun Wang and Doctor Benjian Lv for their valuable comments and suggestions regarding this work.

\end{document}